\documentclass[a4paper,12pt]{amsart}
\usepackage[colorlinks,linkcolor=blue,citecolor=blue]{hyperref}
\usepackage{latexsym, amssymb, amsmath, amscd, amsthm, mathrsfs, bbm, bm}
\usepackage{algorithm, algpseudocode,multirow}
\usepackage[all, knot]{xy}
\xyoption{arc}

\usepackage{anysize}\marginsize{22mm}{22mm}{25mm}{25mm}
\allowdisplaybreaks[4]
\newcommand{\Biggg}{\bBigg@{1.5}}

\def \k {\mathbbm{k}}

\def \dim {\operatorname{dim}}

\numberwithin{equation}{section}
\numberwithin{table}{section}
\numberwithin{equation}{section}
\newtheorem{theorem}{Theorem}[section]

\newtheorem{definition}[theorem]{Definition}
\newtheorem{example}[theorem]{Example}
\newtheorem{remark}[theorem]{Remark}

\begin{document}

\title[Simultaneous direct sum decompositions of several multivariate polynomials]{Simultaneous direct sum decompositions of several multivariate polynomials}
\thanks{Supported by Key Program of Natural Science Foundation of Fujian Province (Grant no. 2024J02018) and National Natural Science Foundation of China (Grant no. 12371037).}



\author{Lishan Fang}
\address{School of Mathematical Sciences, Huaqiao University, Quanzhou 362021, China}
\email{fanglishan@hqu.edu.cn}

\author{Hua-Lin Huang}
\address{School of Mathematical Sciences, Huaqiao University, Quanzhou 362021, China}
\email{hualin.huang@hqu.edu.cn}

\author{Lili Liao}
\address{School of Mathematical Sciences, Huaqiao University, Quanzhou 362021, China}
\email{lili.liao@hqu.edu.cn}

\date{}                                           

\maketitle

\begin{abstract}  
We consider the problem of simultaneous direct sum decomposition of a set of multivariate polynomials. To this end, we extend Harrison's center theory for a single homogeneous polynomial to this broader setting. It is shown that the center of a set of polynomials is a special Jordan algebra, and simultaneous direct sum decompositions of the given polynomials are in bijection with complete sets of orthogonal idempotents of their center algebra. Several examples are provided to illustrate the performance of this method.
\end{abstract}

\vskip 20pt

\section{Introduction}\label{sec:intro}

\subsection{Problem statement}\label{sec:problem}
Throughout, $\k$ is a field and $\k[x_1, x_2, \dots, x_n]$ is the polynomial ring in $n$ variables. 
Let $f_i(x)=f_i(x_1, x_2, \dots, x_n) \in \k[x_1, x_2, \dots, x_n]$ for $1 \le i \le m$. We are concerned about the problem of whether there exists a change of variables $x=Py$ such that 
\begin{equation}\label{eqn:dsd}
    f_i(x)=f_i(Py)= g_{i1}(y_{1} , \ldots , y_{a_{1}}) + g_{i2} (y_{a_{1}+1} , \ldots , y_{a_{2}}) + \cdots + g_{it}(y_{a_{t-1}+1} , \ldots , y_n)
\end{equation}
for all $1 \le i \le m$ with $2 \le t \le n$. If this is the case, determine the appropriate change of variables $x=Py$. This is referred to as the simultaneous direct sum decomposition of several multivariate polynomials. In particular, if $t=n$, then $f_1, f_2, \ldots, f_m$ are simultaneously decomposed as sums of univariate polynomials and this is called the simultaneous diagonalization.
Simultaneous direct sum decompositions may simplify a set of multivariate polynomials synchronously by reducing dimensions and eliminating cross terms. This is of fundamental significance both in theory and in practice involving systems of multivariate polynomials. 

\subsection{Related work}\label{sec:review}
In recent years, the direct sum decomposition of a homogeneous multivariate polynomial has been considered in \cite{buczynska2015apolarity, fedorchuk2020direct, huang2021centres,  wang2015homogeneous} via various approaches. The diagonalization of a homogeneous polynomial was approached in \cite{huang2021diagonalizable, kayal2011efficient, koiran2023absolute, robeva2016orthogonal} and is related to the vast Waring's problem of polynomials. The simultaneous Waring decomposition for a set of homogeneous polynomials was investigated in \cite{carlini2003waring, tiels2013coupled}. In the series of papers \cite{decuyper2019decoupling, dreesen2018decoupling, dreesen2015decoupling}, the authors developed a method to simultaneously decompose a set of nonhomogeneous multivariate real polynomials into linear combinations of univariate polynomials in linear forms of the input variables by tensor decompositions. This is called the decoupling of multivariate polynomials, and the partial decoupling question is proposed as an open problem in \cite{dreesen2015decoupling}. 

The simultaneous direct sum decomposition problem studied in this paper is part of the partial decoupling question. When this problem is restricted to a single homogeneous polynomial, paper \cite{huang2021centres} offers a solution by Harrison's theory of centers \cite{harrison1975grothendieck}. Recall that the center of a homogeneous polynomial $f(x_1, x_2, \dots, x_n)$ can be defined as 
\begin{equation}\label{def:cshp}
        Z(f) := \{X \in \k^{n \times n} \mid (H_fX)^T=H_fX \},  
\end{equation} 
where $H_f$ is the Hessian matrix of $f$. The significance of the center $Z(f)$ of a homogeneous polynomial $f$ of degree at least $3$ is that it has a natural associative algebra structure and direct sum decompositions of 
$f$ are in bijection with complete sets of orthogonal idempotents of $Z(f).$ It also helps to provide both a criterion and an algorithm for the direct sum decomposition of a homogeneous polynomial. 

\subsection{Methods and results}\label{sec:method}
We extend the approach in \cite{huang2021diagonalizable, huang2021centres} to simultaneous direct sum decompositions of an arbitrary set of multivariate polynomials. Firstly, we notice that the center in Equation \eqref{def:cshp} for a single homogeneous polynomial can be naturally generalized to that for a set of not necessarily homogeneous polynomials of various degrees. 

\begin{definition} \label{def:center}
\emph{Let $f_i(x_1, x_2, \dots, x_n) \in \k[x_1, x_2, \dots, x_n]$ for $1 \le i \le m$ be as above. The center of this set of polynomials is defined as
\begin{equation}\label{eqn:center}
Z(f_1, f_2, \dots, f_m) := \{X \in \k^{n \times n} \mid (H_{f_i}X)^T=H_{f_i}X, \ 1 \le i \le m \}.
\end{equation}
}
\end{definition}

Secondly, we observe that $Z(f_1, f_2, \dots, f_m)$ has a natural special Jordan algebra structure and it can be applied to the simultaneous direct sum decompositions of $f_1, f_2, \dots, f_m$ just as the case of a single homogeneous polynomial.
\begin{theorem}\label{thm:intro}
Suppose $f_1, f_2,\ldots, f_m$ are a set of multivariate polynomials in $n$ variables and $Z(f_1, f_2, \dots, f_m)$ is their center. For any $X, Y \in Z(f_1, f_2, \dots, f_m)$, define $X \odot Y=\frac{1}{2}(XY+YX).$
\begin{itemize}
\item[(1)] $(Z(f_1, f_2, \dots, f_m), \odot)$ is a special Jordan algebra.
\item[(2)] There is a one-to-one correspondence between direct sum decompositions of $f_1,f_2,\ldots, f_m$ and complete sets of orthogonal idempotents of $Z(f_1,f_2,\ldots, f_m)$.
\item[(3)] The polynomials $f_1,f_2,\ldots, f_m$ are simultaneously diagonalizable if and only if the center algebra $Z(f_1, f_2, \dots, f_m)$ has a complete set of $n$ orthogonal idempotents.
\end{itemize}
\end{theorem}

Finally, we also provide an algorithm for simultaneous direct sum decompositions of any finite set of multivariate polynomials. This is usually boiled down to some standard tasks in linear algebra.

\subsection{Organization of the paper}\label{sec:organize}

The remainder of this paper is organized as follows.
Section~\ref{sec:center} describes the center theory for a set of multivariate polynomials.
Section~\ref{sec:sdsd} shows the connections between centers and simultaneous direct sum decompositions. It also provides a simple algorithm.
Section~\ref{sec:example} presents more examples to illustrate this technique.

\section{The center of a set of multivariate polynomials}\label{sec:center}

\subsection{Some examples}
We begin with some explicit examples. 

\begin{example}\label{ex:center_single}
\emph{Let $f(x, y, z)=x^{4}+y^{2}+z^{2}$. Then the Hessian matrix of $f(x, y, z)$ is
\begin{align*}
		H =\left (
		\begin{array}{ccc}
			12x^{2}&0&0\\
			0&2&0\\
			0&0&2
		\end{array}
		\right).
\end{align*}
	According to Definition \ref{def:center}, the center $Z(f)$ consists of the following matrices
\begin{align*}
\left(
\begin{array}{ccc}
 a & 0  &  0 \\
0  &  b &  c \\
0  &  c &   d
\end{array}
\right) 	
\end{align*}
for all $a, b, c, d \in \k.$}
\end{example}

\begin{example}\label{ex:center_multiple}
\emph{Let $f_1(u_1, u_2)=54u^3_1 - 54u^2_1u_2 + 8u^2_1 + 18u_1u^2_2 + 16u_1u_2 - 2u^3_2 + 8u_2^2 + 8u_2 + 1$ and $f_2(u_1, u_2)=-27u^3_1 + 27u^2_1u_2 - 24u^2_1 - 9u_1u^2_2 - 48u_1u_2 - 15u_1 + u^3_2 - 24u_2^2 - 19u_2 - 3.$ 
The Hessian matrices of $f_{1}$ and $f_{2}$ are
\begin{align*}
    H_{f_1}=\left(
    \begin{array}{cc}
    		324u_{1}-108u_{2}+16 & -108u_{1}+36u_{2}+16 \\
    		-108u_{1}+36u_{2}+16 & 36u_{1}-12u_{2}+16 \\
    	\end{array}
    	\right),
     \end{align*}
\begin{align*}
   H_{f_2}=\left(
    	\begin{array}{cc}
    		-162u_{1}+54u_{2}-48 & 54u_{1}-18u_{2}-48\\
    		54u_{1}-18u_{2}-48& -18u_{1}+6u_{2}-48\\
    	\end{array}
    	\right),
\end{align*}
respectively. Then by a direct computation, one has 
\begin{align*}
Z(f_1, f_2)=\left\{ 
\left(
\begin{array}{cc}
  a   &  \frac{1}{2}(b-a) \\
  \frac{3}{2}(b-a)   &    b
\end{array}
\right)
\Big|
a, b \in \k
 \right\}.
\end{align*}}
\end{example}

In Example \ref{ex:center_multiple}, the center is a commutative associative subalgebra of the full matrix algebra. While in Example \ref{ex:center_single}, the center is not even closed under the usual matrix multiplication. Therefore, we need to consider on the center of a set of polynomials a new algebraic structure which contains that for a single homogeneous polynomial of degree $>2$ as a special case. It turns out that Jordan algebras will do the job.  
 
\subsection{Jordan algebras}
For the convenience of the reader, we recall the definition of Jordan algebras, see e.g. \cite{mccrimmon2003taste} for more detail. 
\begin{definition}\label{def:jordan}
\emph{A \textit{Jordan algebra} consists of a vector space $J$ equipped with a bilinear product $X \odot Y$ satisfying the commutative law and the Jordan identity:
\begin{equation}
    X \odot Y =Y \odot X,  \quad    (X^{2} \odot Y) \odot X =X^{2}\odot (Y \odot X).
\end{equation}
Any associative algebra $A$ gives rise to a Jordan algebra $A^{+}$ under symmetrized product 
\begin{equation*}
X \odot Y=\frac{1}{2}(XY+YX).
\end{equation*}
A Jordan algebra is said to be \textit{special} if it can be realized as a Jordan subalgebra of some $A^+$.
An element $e \in J$ is called an idempotent if $e \odot e=e.$ Two idempotents $e_1$ and $e_2$ are said to be orthogonal, if $e_1 \odot e_2=0.$ We call $e_1, e_2, \dots, e_s$ a complete set of orthogonal idempotents in a unital Jordan algebra $J$ if 
\begin{equation*}
e_i \odot e_i =e_i, \forall i, \quad e_i \odot e_j=0, \forall i \ne j, \quad \sum_{i=1}^s e_i =1.
\end{equation*}
}
\end{definition}

\begin{remark}\label{rem:idempotent}
\emph{In a unital Jordan subalgebra $J$ of $({\k^{n \times n}})^+,$ if $e_1, e_2, \dots, e_s$ are a complete set of orthogonal idempotents, then we have
\begin{equation*}
e_i^2=e_i, \forall i, \quad e_ie_j=0, \forall i \ne j, \quad \sum_{i=1}^se_i=I_n
\end{equation*}
in the usual operations of matrices. Only $e_ie_j=0$ needs an explanation. Since $e_i^2=e_i,$ there exists an invertible matrix $P$ such that $P^{-1}e_iP=\begin{pmatrix}
    I & 0 \\
    0 & 0
\end{pmatrix}.$ Take the conformal block form of $P^{-1}e_jP$ as $\begin{pmatrix}
    X & Y \\
    Z & W
\end{pmatrix}.$ Then by definition $e_i \odot e_j=\frac{1}{2}(e_ie_j+e_je_i)=0,$ from this we can easily deduce $X=0, Y=0, Z=0.$ This implies $e_i e_j=0.$      }
\end{remark}

\subsection{Main properties of centers} Now we are ready to state the main structural result for center algebras.
\begin{theorem}\label{thm:property}
Let $Z(f_1, f_2, \ldots, f_m)$ be the center of $f_1, f_2, \ldots, f_m \in \k[x_1, x_2, \dots, x_n]$.
\begin{itemize}
\item[(1)] $(Z(f_1, f_2, \ldots, f_m), \odot)$ is a special Jordan algebra.
\item[(2)] If the coefficients of $f_1, f_2, \ldots, f_m$ are generic and at least one of the $f_i$ has degree not less than $3,$ then $(Z(f_1, f_2, \ldots, f_m), \odot) \cong \k$. 
\end{itemize}
\end{theorem}
\begin{proof}
(1) For any $X, Y \in Z(f_1, f_2, \ldots, f_m)$, we have $X \odot Y=\frac{1}{2}(XY+YX)$. Then it follows that
\begin{equation*}
    H_{f_i}(X \odot Y) = H_{f_i}\frac{(XY + YX)}{2} = \frac{(Y^TX^T + X^TY^T)}{2}H_{f_i} = [H_{f_i}(X \odot Y)]^T,
\end{equation*}
for $1\le i \le m$. 
In other words, $(Z(f_1, f_2, \ldots, f_m), \odot)$ is closed under $\odot$ and hence a Jordan subalgebra of ${(\k^{n \times n})}^+$.

(2) Without loss of generality, we assume $f_1$ has at least two variables of degree $d \ge 3$. If $f_1$ has more than two variables or has degree $d>3$, then a generic $f_1=\sum_{l=0}^d f_1^{(l)}$ has nondegenerate degree $d$ homogeneous component $f_1^{(d)}$. By \cite[Theorem 3.2]{huang2021centres}, $Z(f_1) \cong \k$ for a generic $f_1$ as it is clear by Definition \ref{def:center} that $Z(f_1)=\cap_{l=0}^d Z(f_1^{(l)}).$ If $f_1$ is a binary polynomial of degree $3,$ say 
\begin{equation*}
		f_1(x_1, x_2) = Ax_1^{3} + 3Bx_1^{2}x_2 + 3Cx_1x_2^{2} + Dx_2^{3} + 3Ex_1^{2} +6Fx_1x_2 + 3Gx_2^{2} + 3Hx_1 + 3Ix_2 + K ,
\end{equation*}
	where $A,B,C,D,E,F,G,H,I,K \in \k,$ then the center $Z(f_1)$ is the solution of the following linear equations
\begin{align*}
		\left\{
		\begin{array}{ccccc}
			Az_{12}& + &B(z_{22} - z_{11})& - &Cz_{21}=0, \\
			Bz_{12}& + &C(z_{22} - z_{11})& + &Dz_{21}=0, \\      
			Ez_{12}& + &F(z_{22} - z_{11})& + &Gz_{21}=0.
		\end{array}
		\right .
\end{align*}
So for generic $A,B,C,D,E,F,G,$ the previous linear equations have trivial solutions, namely, $Z(f_1)$ consists of scalar matrices. That is, $Z(f_1) \cong \k.$ It is obvious that $Z(f_1, f_2, \dots, f_m)=\cap_{i=1}^m Z(f_i) \cong \k$ for a general set of multivariate polynomials $f_1, f_2, \ldots, f_m$ with the prescribed conditions. 
\end{proof}

\begin{remark}
\emph{ The adjective “generic” in the statement of Theorem \ref{thm:property} is to be understood in the usual sense of algebraic geometry, that is, there exists a polynomial in the coefficients of $f_1, f_2, \dots, f_m$ such that the conclusion holds for all polynomials $f_1, f_2, \dots, f_m$ at which that polynomial does not vanish. In particular, it will only fail on a set of instances that has measure zero. }  
\end{remark}

\section{Simultaneous direct sum decompositions of polynomials}\label{sec:sdsd}




\subsection{Center and simultaneous direct sum decompositions}\label{sec:center_decomp}

For a set of multivariate polynomials, we apply their center algebra to simultaneously decompose them as a sum of polynomials in disjoint sets of variables.

\begin{theorem}\label{thm:center}
Suppose $f_1, f_2, \ldots, f_m \in \k[x_1, x_2, \dots, x_n].$ Then
\begin{itemize}\label{thm:idm}
    \item[{(1)}] There is a one-to-one correspondence between simultaneous direct sum decompositions of $f_1$, $f_2$, $\ldots$, $f_m$ and complete sets of orthogonal idempotents of $Z(f_1, f_2, \ldots, f_m)$.
    \item[{(2)}] $f_1, f_2, \ldots, f_m$ are not simultaneously decomposable if and only if $Z(f_1$, $f_2$, $\ldots$, $f_m)$ has no nontrivial idempotent elements.
    \item[{(3)}] $f_1, f_2, \ldots, f_m$ are simultaneously diagonalizable if and only if $Z(f_1, f_2, \ldots, f_m)$ has a complete set of $n$ orthogonal idempotents.
    \item[{(4)}] If the coefficients of $f_1, f_2, \ldots, f_m$ are generic and at least one of the $f_i$ has degree not less than $3,$ then $f_1, f_2, \ldots, f_m$ are not simultaneously decomposable. 
\end{itemize}
\end{theorem}



\begin{proof}
We include a proof of item (1) below, which clearly implies items (2) and (3), and (4) with the help of Theorem~\ref{thm:property}. As preparation, firstly we consider how the center varies with variable substitution. Take the change of variables $x = Py$ and denote $g_i(y)=f_i(Py)$ for $1 \le i \le m$. Let $H_i$ and $G_i$ denote the Hessian matrices of $f_i$ and $g_i$ with respect to the variables $x$ and $y$, respectively. It is easy to see that they satisfy $G_i=P^T H_i P$. Since $(G_iY)^T=G_iY$ for any $Y \in Z(g_1,g_2,\ldots, g_m)$, we have $(G_iY)^T = Y^{T}P^{T}H_iP=P^{T}H_iPY$, which implies that $PYP^{-1} \in Z(f_1,f_2,\ldots, f_m)$. It follows that $Z(g_1,g_2,\ldots, g_m)=P^{-1}Z(f_1,f_2,\ldots, f_m)P.$ 


Suppose there exists a simultaneous direct sum decomposition of $f_1$, $f_2$, $\ldots$, $f_m$ given by $g_i(y):=f_i(Py) = g_{i1}(y_{1} , \ldots , y_{a_{1}}) + g_{i2} (y_{a_{1}+1} , \ldots , y_{a_{2}}) + \cdots + g_{it}(y_{a_{t-1}+1} , \ldots , y_{n})$ for $1 \le i \le m$. It is obvious that $H_{g_i}$ is block-diagonal, written as $H_{g_{i1}}\oplus H_{g_{i2}}\oplus \cdots \oplus H_{g_{it}}$. We then have $Z(g_1, g_2, \ldots, g_m) = Z(g_{11}, g_{21}, \ldots, g_{m1}) \oplus Z(g_{12}, g_{22}, \ldots, g_{m2}) \oplus \cdots \oplus Z(g_{1t}, g_{2t}, \ldots, g_{mt})$ by Equation~\eqref{eqn:center}. It is clear that the identity matrix $I_j$ of $Z(g_{1j}, g_{2j}, \ldots, g_{mj})$ corresponds to an idempotent $E_j$ of $Z(g_1, g_2, \dots, g_m),$ where $ E_{j}$ takes the form of a block matrix with $I_j$ on the diagonal and 0 elsewhere
\begin{align*}
			E_{j}=	\left (
			\begin{array}{cccc}
				\bm{0}&&\\
				& I_j &\\
				& & \bm{0}\\
			\end{array}
			\right)
\end{align*}
for $1 \le j \le t$. Note that $Z(f_1, f_2, \dots, f_m)=PZ(g_1, g_2, \dots, g_m)P^{-1},$ so $PE_1P^{-1}, PE_2P^{-1}, \ldots,$ $ PE_mP^{-1}$ are a complete set of orthogonal idempotents of $Z(f_1, f_2, \ldots, f_m)$.


Conversely, suppose there exists a complete set of orthogonal idempotents $\epsilon_1, \epsilon_2, \ldots, \epsilon_t$ in $Z(f_1, f_2, \ldots, f_m).$ By Remark \ref{rem:idempotent}, we have $\epsilon_i^2=\epsilon_i$ for all $i$ and $\epsilon_i\epsilon_j=0$ for all $i \ne j$ as matrices. Since the $\epsilon_i$'s are diagonalizable matrices and they mutually commute, $\epsilon_1, \epsilon_2, \ldots ,\epsilon_t$ are simultaneously diagoanalizable. 
Thus, there exists some invertible matrix $P \in \k^{n \times n}$ such that $P^{-1}\epsilon_j P=E_j$ as above and $\sum_{j=1 }^{t}E_{j}=I_n$.
Take the change of variables $x=Py$ and let $g_i$ and $G_i$ be as above. We now have $G_iE_j=(G_iE_j)^T=E_j^TG_{i}^T=E_jG_i$ for $1\le i \le m$ and $1\le j \le t$. It follows that the $G_i$'s are quasi-diagonal  and $g_1,g_2,\ldots, g_m$ can be written as sums of polynomials in disjoint sets of variables
\begin{equation*}
    g_i(y_1, y_2, \ldots, y_n) = g_{i1}(y_{1} , \ldots , y_{a_{1}}) + g_{i2} (y_{a_{1}+1} , \ldots , y_{a_{2}}) + \cdots + g_{it}(y_{a_{t-1}+1} , \ldots , y_n)
\end{equation*}
for all $1 \le i \le m$. Therefore, there is a simultaneous direct sum decomposition of $f_1, f_2, \ldots, f_m$ corresponding to the complete set of orthogonal idempotents $\epsilon_1, \epsilon_2, \ldots, \epsilon_t$.
\end{proof}

\subsection{Algorithm for simultaneous direct sum decompositions}\label{sec:decomp_alg}
According to the proof of Theorem \ref{thm:center},
we provide an algorithm for the simultaneous direct sum decomposition of a set of multivariate polynomials in Algorithm~\ref{alg:sdsd}. This algorithm contains three steps, which are compute center, compute idempotents and decompose polynomials. It decomposes all given polynomials into two components in disjoint sets of variables, which are output as two sets of polynomials. Since these polynomials may still be simultaneously decomposable, they will be used as input polynomials of Algorithm~\ref{alg:sdsd} and decomposed recursively until output polynomials are not simultaneously decomposable.


\begin{algorithm}
    \caption{Simultaneous direct sum decomposition of multivariate polynomials}\label{alg:sdsd}
    \begin{algorithmic} [1]
    \Statex \textbf{Input}: polynomials~$f_1, f_2, \ldots, f_m$
    \Statex \textbf{Output}: decomposed polynomials $g_{11}, \ldots, g_{m1}$ and $g_{12}, \ldots, g_{m2}$
    \State Compute center. Solve system $(H_{f_i}X)^T=H_{f_i}X$ for $1 \le i \le m$ to compute the center~$Z(f_1, f_2,\ldots, f_m)$ and obtain base matrices $X_{j}$ for $1 \le j \le \dim{Z(f_1, f_2,\ldots, f_m)}$. 
    \State Compute idempotents. Compute orthogonal idempotents~$\epsilon$ using $Z(f_1, f_2, \ldots, f_m)$. If a nontrivial $\epsilon$ does not exist, $f_1, f_2, \ldots, f_m$ are simultaneously indecomposable. Otherwise, continue.
    \State Decompose polynomials. Decompose~$f_1, f_2, \ldots, f_m$ using~$\epsilon$ and separate variables such that $f_{i}(x_1, \ldots, x_n) = g_{i1}(x_1, \ldots, x_k) + g_{i2}(x_{k+1}, \ldots, x_n)$ for $1 \le i \le m$ and $1 < k < n$. Go to step 1 and continue to decompose $g_{11}, g_{21},\ldots, g_{m1}$ and $g_{12}, g_{22}, \ldots, g_{m2}$ separately.
    \end{algorithmic}
\end{algorithm}

The compute center step needs to solve a system of linear equations, which is sparse and contains many redundant equations. The complexity analysis in~\cite{fang2023numerical} can be easily extended to multiple polynomials, which shows that this step has polynomial complexity. The compute idempotents step requires solving systems of nonlinear equations similar to the direct sum decomposition of one homogeneous polynomial proposed in~\cite{huang2021centres}, which may be computationally expensive. However, it only needs to solve a system of sparse quadratic equations if $f_1, f_2, \ldots, f_m$ are simultaneously diagonalizable. It is shown that the computation is performed efficiently in numerical experiments~\cite{fang2023numerical}. The decompose polynomials step takes the same change of variables for each polynomial, which can be executed efficiently in parallel.

\section{Some examples}\label{sec:example}

In this section, we provide three examples to demonstrate the performance of the proposed technique. Example~\ref{ex:1}, taken from \cite{dreesen2015decoupling}, shows the simultaneous diagonalization of two nonhomogeneous polynomials in two variables using centers. Example~\ref{ex:2} describes the simultaneous direct sum decomposition of two polynomials in four variables. Example~\ref{ex:3} includes three polynomials that are decomposable as direct sums separately, but not simultaneously decomposable. Note that while the three examples only include real polynomials, this technique works for polynomials over an arbitrary field.

\begin{example}\label{ex:1}
\emph{(Example~\ref{ex:center_multiple} continued)
Let $f_1(u_1, u_2)=54u^3_1 - 54u^2_1u_2 + 8u^2_1 + 18u_1u^2_2 + 16u_1u_2 - 2u^3_2 + 8u_2^2 + 8u_2 + 1$ and $f_2(u_1, u_2)=-27u^3_1 + 27u^2_1u_2 - 24u^2_1 - 9u_1u^2_2 - 48u_1u_2 - 15u_1 + u^3_2 - 24u_2^2 - 19u_2 - 3.$}

\emph{We have 
    \begin{align*}
    Z(f_1, f_2)=\left\{ 
    \left(
    \begin{array}{cc}
      a   &  \frac{1}{2}(b-a) \\
      \frac{3}{2}(b-a)   &    b
    \end{array}
    \right)
    \Big|
    a, b \in \k
     \right\}
      \end{align*}
    and the two base matrices of $Z(f_{1},f_{2})$ are
      \begin{align*}
   	X_{1} =	\left (
   	\begin{array}{cc}
   		1&-\frac{1}{2}\\
   		-\frac{3}{2}&0
   	\end{array}
   	\right),\quad
   	X_{2} =	\left (
   	\begin{array}{cc}
   		0&\frac{1}{2}\\
   		\frac{3}{2}&1
   	\end{array}
   	\right).
 \end{align*}
    In light of the structure of $Z(f_{1},f_{2})$, it is not hard to work out a pair of orthogonal idempotents
      \begin{align*}
   	\epsilon_{1} =\left (
   	\begin{array}{cc}
   		\frac{1}{4}&\frac{1}{4}\\
   		\frac{3}{4}&\frac{3}{4}
   	\end{array}
   	\right),\quad
   	\epsilon_{2} =	\left (
   	\begin{array}{cc}
   		\frac{3}{4}&-\frac{1}{4}\\
   		-\frac{3}{4}&\frac{1}{4}
   	\end{array}
   	\right)	.
  \end{align*}
}

\emph{According to the proof of Theorem~\ref{thm:idm}, since there exists an invertible matrix $P=\left (
   	\begin{array}{cc}
   		-\frac{1}{8} & \frac{1}{4} \\
   		-\frac{3}{8} & -\frac{1}{4}
   	\end{array}
   	\right)$
such that
  \begin{align*}
   	P^{-1}\epsilon_{1}P =	\left (
   	\begin{array}{cc}
   		1 & 0 \\
   		0 & 0
   	\end{array}
   	\right),\quad
   	P^{-1}\epsilon_{2}P =	\left (
   	\begin{array}{cc}
   		0 & 0 \\
   		0 & 1
   	\end{array}
   	\right),
  \end{align*}
we take the change of variables
\begin{equation*}
   u_{1} = -\frac{1}{8}x_{1} + \frac{1}{4}x_{2}, \quad u_{2} = -\frac{3}{8}x_{1} -\frac{1}{4}x_{2}.
\end{equation*}
Then $f_{1}$ and $f_{2}$ are simultaneously diagonalized as
\begin{align*}
   f_{1}\left(-\frac{1}{8}x_{1} + \frac{1}{4}x_{2}, -\frac{3}{8}x_{1} - \frac{1}{4}x_{2}\right) &= 2x_{1}^{2}-3x_{1}+2x_{2}^{3}-2x_{2}+1, \\
    f_{2}\left(-\frac{1}{8}x_{1} + \frac{1}{4}x_{2}, -\frac{3}{8}x_{1} - \frac{1}{4}x_{2}\right) &=-6x_{1}^{2}+9x_{1}-x_{2}^{3}+x_{2}-3.
\end{align*}
}
\end{example}

\begin{example}\label{ex:2}
\emph{Let $f_{1}(x_{1},x_{2},x_{3},x_{4})=x_{1}^{3}+3x_{1}^{2}x_{2}+3x_{1}^{2}x_{3}+3x_{1}x_{2}^{2}+6x_{1}x_{2}x_{3}+3x_{1}x_{3}^{2}+2x_{2}^{3}+6x_{2}x_{3}^{2}+x_{3}^{2}x_{4}+x_{4}^{2}+2x_{3}+1$ and $f_{2}(x_{1},x_{2},x_{3},x_{4})=2x_{1}^{3}+6x_{1}^{2}x_{2}+6x_{1}^{2}x_{3}+6x_{1}x_{2}^{2}+12x_{1}x_{2}x_{3}+6x_{1}x_{3}^{2}+5x_{2}^{3}-3x_{2}^{2}x_{3}+15x_{2}x_{3}^{2}-x_{3}^{3}+x_{3}x_{4}^{2}+3x_{4}$.}

\emph{We calculate the Hessian matrices of $f_1$ and $f_2$ and compute their center by Equation~\eqref{eqn:center} as
  \begin{align*}
Z(f_1, f_2)=\left\{ 
    \left(
    \begin{array}{cccc}
      a+b+c & a+2c & a & 0 \\
     0 & b-c & c& 0\\
     0 & 0 & b & 0\\
     0 & 0 & 0 & b
    \end{array}
    \right)
    \Bigg|
    a, b, c \in \k
     \right\}.
  \end{align*}
The three base matrices of $Z(f_1, f_2)$ can be obtained as follows
  \begin{align*}
   	X_{1} =	\left (
   	\begin{array}{cccc}
   		1 & 0 & 0 & 0 \\
     0 & 1 & 0 & 0\\
     0 & 0 & 1 & 0\\
     0 & 0 & 0 & 1
   	\end{array}
   	\right),\quad
   	X_{2} =	\left (
   	\begin{array}{cccc}
     1 & 1 & 1 & 0 \\
     0 & 0 & 0 & 0\\
     0 & 0 & 0 & 0\\
     0 & 0 & 0 & 0
   	\end{array}
   	\right),\quad
   	X_{3} =	\left (
   	\begin{array}{cccc}
     1 & 2 & 0 & 0 \\
     0 & -1 & 1 & 0\\
     0 & 0 & 0 & 0\\
     0 & 0 & 0 & 0
   	\end{array}
   	\right).
      \end{align*}
Then we have the following orthogonal idempotents
  \begin{align*}
   	\epsilon_{1} =	\left (
   	\begin{array}{cccc}
   		1 & 1 & 1 & 0 \\
     0 & 0 & 0 & 0\\
     0 & 0 & 0 & 0\\
     0 & 0 & 0 & 0
   	\end{array}
   	\right),\quad
   	\epsilon_{2} =	\left (
   \begin{array}{cccc}
     0 & -1 & 1 & 0 \\
     0 & 1 & -1 & 0\\
     0 & 0 & 0 & 0\\
     0 & 0 & 0 & 0
   	\end{array}
   	\right),\quad
   	\epsilon_{3} =	\left (
   \begin{array}{cccc}
     0 & 0 & -2 & 0 \\
     0 & 0 & 1 & 0\\
     0 & 0 & 1 & 0\\
     0 & 0 & 0 & 1
   	\end{array}
   	\right)
  \end{align*}
and 
  \begin{align*}
   	P^{-1}\epsilon_{1}P =	\left (
   	\begin{array}{cccc}
   		1 & 0 & 0 & 0 \\
     0 & 0 & 0 & 0\\
     0 & 0 & 0 & 0\\
     0 & 0 & 0 & 0
   	\end{array}
   	\right),\quad
   	P^{-1}\epsilon_{2}P =	\left (
   \begin{array}{cccc}
     0 & 0 & 0 & 0 \\
     0 & 1 & 0 & 0\\
     0 & 0 & 0 & 0\\
     0 & 0 & 0 & 0
   	\end{array}
   	\right),\quad
   	P^{-1}\epsilon_{3}P =	\left (
   \begin{array}{cccc}
     0 & 0 & 0 & 0 \\
     0 & 0 & 0 & 0\\
     0 & 0 & 1 & 0\\
     0 & 0 & 0 & 1
   	\end{array}
   	\right),
  \end{align*}
where
  \begin{align*}
   	P =	 \left (
   \begin{array}{cccc}
     1 & -1 & -2 & 0 \\
     0 & 1 & 1 & 0\\
     0 & 0 & 1 & 0\\
     0 & 0 & 0 & 1
   	\end{array}
   	\right).
  \end{align*}
Take the change of variables
\begin{equation*}
   x_{1} = y_{1} - y_{2} - 2y_{3}, \quad x_{2} = y_{2} + y_{3},\quad x_{3}=y_{3},\quad x_{4}=y_{4}
\end{equation*}
and then $f_{1}$ and $f_{2}$ are simultaneously decomposed as
\begin{align*}
    f_{1}\left(y_{1}-y_{2}-2y_{3},y_{2}+y_{3},y_{3},y_{4}\right) &= y_{1}^{3}+y_{2}^{3}+y_{3}^{2}y_{4}+y_{4}^{2}+2y_{3}+1, \\
    f_{2}\left(y_{1}-y_{2}-2y_{3},y_{2}+y_{3},y_{3},y_{4}\right) &=2y_{1}^{3}+3y_{2}^{3}+y_{3}y_{4}^{2}+3y_{4}.
\end{align*}
}
\end{example}

\begin{example}\label{ex:3}
\emph{Let $f_{1}(x_{1},x_{2},x_{3})=81x_{1}^{4}+108x_{1}^{3}x_{3}+54x_{1}^{2}x_{3}^{2}+12x_{1}x_{3}^{3}+x_{3}^{4}+x_{2}^{3}+x_{3}^{3}+x_{2}x_{3}^{2}+2x_{2}^{2}+5x_{3}+1$,\quad $f_{2}(x_{1},x_{2},x_{3})=x_{1}^{3}-6x_{1}^{2}x_{2}+12x_{1}x_{2}^{2}-7x_{2}^{3}+3x_{2}x_{3}+7x_{1}+5$ and $f_{3}(x_{1},x_{2},x_{3})=27x_{2}^{3}+36x_{2}x_{3}^{2}-54x_{2}^{2}x_{3}-8x_{3}^{3}+5x_{1}^{3}+2x_{1}^{2}+7x_{1}+12x_{2}-8x_{3}+5$. We firstly compute the direct sum decomposition of $f_1$, $f_2$ and $f_3$, separately. We can calculate the Hessian matrix of $f_1$ as
\begin{align*}
    H_{1}=\left(
    \begin{array}{ccc} 		972x_{1}^{2}+648x_{1}x_{3}+108x_{3}^{2} & 0 &324x_{1}^{2}+216x_{1}x_{3}+36x_{3}^{2} \\
    		0 & 6x_{2}+4 &2x_{3} \\  324x_{1}^{2}+216x_{1}x_{3}+36x_{3}^{2} & 2x_{3} &108x_{1}^{2}+72x_{1}x_{3}+12x_{3}^{2}+6x_{3}+2x_{2} 
    	\end{array}
    	\right).
      \end{align*}
By Equation~\eqref{eqn:center}, we have
 \begin{align*}
 Z(f_1)=\left\{ 
    \left(
    \begin{array}{ccc}
      a+3b & 0 & b  \\
     0 & a &  0\\
     0 & 0 & a
    \end{array}
    \right)
    \Bigg|
    a, b \in \k
     \right\}
     \end{align*}
   and the two base matrices of $Z(f_1)$ are
\begin{align*}
   	X_{1} =	\left (
   	\begin{array}{ccc}
   		1 & 0 & 0 \\
     0 & 1 & 0\\
     0 & 0 & 1
   	\end{array}
   	\right),\quad
   	X_{2} =	\left (
   	\begin{array}{ccc}
    3 & 0 & 1\\
     0 & 0 & 0\\
     0 & 0 & 0 
   	\end{array}
   	\right).
\end{align*}
    We then compute a pair of orthogonal idempotents
\begin{align*}
   	\epsilon_{1} =	\left (
   \begin{array}{ccc}
    	1 & 0 & \frac{1}{3} \\
     0 & 0 & 0\\
     0 & 0 & 0 
   	\end{array}
   	\right),\quad
   	\epsilon_{2} =	\left (
   \begin{array}{ccc}
   		0 & 0 & -\frac{1}{3} \\
     0 & 1 & 0\\
     0 & 0 & 1 
   	\end{array}
   	\right)
\end{align*}
and
\begin{align*}
   	P^{-1}\epsilon_{1}P =	\left (
   \begin{array}{ccc}
     1 & 0 & 0 \\
     0 & 0 & 0\\
     0 & 0 & 0 
   	\end{array}
   	\right),\quad
   	P^{-1}\epsilon_{2}P =	\left (
   \begin{array}{ccc}
     0 & 0 & 0 \\
     0 & 1 & 0\\
     0 & 0 & 1 
   	\end{array}
   	\right),
\end{align*}
where
\begin{align*}
   	P =	\left (
   \begin{array}{ccc}
     \frac{1}{3} & 0 & -\frac{1}{3} \\
     0 & 1 & 0\\
     0 & 0 & 1 
   	\end{array}
   	\right).
\end{align*}
}
   
\emph{Take the change of variables
\begin{equation*}
   x_{1} = \frac{1}{3}y_{1} - \frac{1}{3}y_{3}, \quad x_{2} = y_{2} ,\quad x_{3} = y_{3},
\end{equation*}
then $f_{1}$ is decomposed as
\begin{equation*}
   f_{1}\left(\frac{1}{3}y_{1} - \frac{1}{3}y_{3},y_{2},y_{3}\right) = y_{1}^{4}+y_{2}^{3}+y_{3}^{3}+y_{2}y_{3}^{2}+2y_{2}^{2}+5y_{3}+1.
\end{equation*}
We can calculate the Hessian matrix of $f_2$ as
\begin{equation*}
    H_{2}=\left(
    \begin{array}{ccc}
    		6x_{1}-12x_{2} & -12x_{1}+24x_{2} &0 \\
    		-12x_{1}+24x_{2} & 24x_{1}-42x_{2} &3 \\
             0 & 3 &0 
    	\end{array}
    	\right).
\end{equation*}
By Equation~\eqref{eqn:center}, we have
 \begin{align*}
 Z(f_2)=\left\{ 
    \left(
    \begin{array}{ccc}
      a-\frac{1}{2}b & b & 0  \\
     0 & a &  0\\
     0 & c & a
    \end{array}
    \right)
    \Bigg|
    a,b,c \in \k
     \right\}
     \end{align*}
   and the three base matrices of $Z(f_2)$ can be obtained as follows
\begin{align*}
   	X_{3} =	\left (
   	\begin{array}{ccc}
   		1 & 0 & 0 \\
     0 & 1 & 0\\
     0 & 0 & 1
   	\end{array}
   	\right),\quad
   	X_{4} =	\left (
   	\begin{array}{ccc}
    	-\frac{1}{2} & 1 & 0 \\
     0 & 0 & 0\\
     0 & 0 & 0 
   	\end{array}
   	\right),\quad
   	X_{5} =	\left (
   	\begin{array}{ccc}
    	0 & 0 & 0 \\
     0 & 0 & 0\\
     0 &1 & 0 
   	\end{array}
   	\right).
 \end{align*}
    The corresponding orthogonal idempotents are calculated as
\begin{align*}
   	\epsilon_{3} =	\left (
   \begin{array}{ccc}
    	1 & -2 & 0 \\
     0 & 0 & 0\\
     0 & 0 & 0 
   	\end{array}
   	\right),\quad
   	\epsilon_{4} =	\left (
   \begin{array}{ccc}
   		0 & 2 & 0 \\
     0 & 1 & 0\\
     0 & 0 & 1 
   	\end{array}
   	\right)
\end{align*}
and
\begin{align*}
   	P^{-1}\epsilon_{3}P =	\left (
   \begin{array}{ccc}
    	1 & 0 & 0 \\
     0 & 0 & 0\\
     0 & 0 & 0 
   	\end{array}
   	\right),\quad
   	P^{-1}\epsilon_{4}P =	\left (
   \begin{array}{ccc}
   		0 & 0 & 0 \\
     0 & 1 & 0\\
     0 & 0 & 1 
   	\end{array}
   	\right),
\end{align*}
where
\begin{align*}
   	P = \left (
   \begin{array}{ccc}
   		1 & 2 & 0 \\
     0 & 1 & 0\\
     0 & 0 & 1 
   	\end{array}
   	\right).
\end{align*}
}
   
 \emph{According to the proof of Theorem~\ref{thm:idm}, take the change of variables
\begin{equation*}
   x_{1} = z_{1}+2z_{2}, \quad x_{2} = z_{2} ,\quad x_{3}= z_{3}.
\end{equation*}
Then $f_{2}$ is decomposed as
\begin{equation*}
   f_{2}(z_{1}+2z_{2},z_{2},z_{3}) = z_{1}^{3}+z_{2}^{3}+3z_{2}z_{3}+7z_{1}+14z_{2}+5.
\end{equation*}
We calculate the Hessian matrix of $f_3$ as
\begin{align*}
    H_{3}=\left(
    \begin{array}{ccc}
    		30x_{1}+4 & 0 &0 \\
    		0 & 162x_{2}-108x_{3} & -108x_{2}+72x_{3} \\
             0 & -108x_{2}+72x_{3} & 72x_{2}-48x_{3} 
    	\end{array}
    	\right).
\end{align*}
By Equation~\eqref{eqn:center}, we have
\begin{align*}
 Z(f_3)=\left\{ 
    \left(
    \begin{array}{ccc}
      a & 0 & 0  \\
     e & d &  b\\
     \frac{3}{2}e & c & d-\frac{2}{3}c+\frac{3}{2}b
    \end{array}
    \right)
    \Bigg|
    a, b,c,d,e \in \k
     \right\}
     \end{align*}
     A pair of orthogonal idempotents are determined using $Z(f_{3})$ as
\begin{align*}
   	\epsilon_{5} =	\left (
   \begin{array}{ccc}
    	1 & 0 & 0 \\
     0 & 0 & \frac{2}{3}\\
     0 & 0 & 1 
   	\end{array}
   	\right),\quad
   	\epsilon_{6} =	\left (
   \begin{array}{ccc}
   		0 & 0 & 0 \\
     0 & 1 & -\frac{2}{3}\\
     0 & 0 & 0 
   	\end{array}
   	\right).
\end{align*}
and
\begin{align*}
   	P^{-1}\epsilon_{5}P =	\left (
   \begin{array}{ccc}
     1 & 0 & 0 \\
     0 & 0 & 0 \\
     0 & 0 & 1 
   	\end{array}
   	\right),\quad
   	P^{-1}\epsilon_{6}P =	\left (
   \begin{array}{ccc}
     0 & 0 & 0 \\
     0 & 1 & 0 \\
     0 & 0 & 0 
   	\end{array}
   	\right),
\end{align*}
where
\begin{align*}
   	P =	\left (
   \begin{array}{ccc}
   		1 & 0 & 0 \\
     0 & \frac{1}{3} & \frac{2}{3}\\
     0 & 0 & 1 
   	\end{array}
   	\right).
\end{align*}
   }
\emph{According to the proof of Theorem~\ref{thm:idm}, take the change of variables
\begin{equation*}
   x_{1} = u_{1}, \quad x_{2} = \frac{1}{3}u_{2}+\frac{2}{3}u_{3} ,\quad x_{3}=u_{3},
\end{equation*}
then $f_{3}$ is decomposed as
\begin{equation*}
   f_{3}\left(u_{1},\frac{1}{3}u_{2}+\frac{2}{3}u_{3},u_{3}\right) = 5u_{1}^{3}+u_{2}^{3}+2u_{1}^{2}+4u_{2}+7u_{1}+5.
\end{equation*}
}

\emph{Each of $f_1$, $f_2$ and $f_3$ is decomposable as a direct sum individually. However, if we calculate their Hessian matrices and center by Equation~\eqref{eqn:center}, we have
 \begin{align*}
 Z(f_{1},f_{2},f_{3})=\left\{ 
    \left(
    \begin{array}{ccc}
     a & 0 & 0  \\
     0 & a &  0\\
     0 & 0 & a
    \end{array}
    \right)
    \Bigg|
    a\in \k
     \right\},
     \end{align*}
which does not contain any nontrivial idempotent elements. Therefore, $f_{1},f_{2},f_{3} $ are not simultaneously decomposable as direct sums.}
\end{example}


\section*{Declaration of interests}
The authors declare that they have no known competing financial interests or personal relationships that could have appeared to influence the work reported in this paper.

\bibliographystyle{plainnat}




\end{document}